\documentclass[11pt, reqno]{amsart}
\usepackage{amsmath, amsthm, amscd, amsfonts, amssymb, graphicx, color}
\usepackage[bookmarksnumbered, colorlinks, plainpages]{hyperref}
\usepackage{mathrsfs}
\newtheorem{theorem}{Theorem}[section]

\newtheorem{proposition}[theorem]{Proposition}
\newtheorem{corollary}[theorem]{Corollary}
\theoremstyle{definition}
\newtheorem{definition}[theorem]{Definition}
\newtheorem{example}[theorem]{Example}

\theoremstyle{remark}
\newtheorem{remark}[theorem]{Remark}
\numberwithin{equation}{section}

\begin{document}

 \title[Minimal Wittstock Envelopes]{Minimal Wittstock Envelopes  of Hermitian Completely Bounded Maps}

\author[  M. Kian]{  Mohsen Kian}

\address{Mohsen Kian:\, Department  of Mathematics,  University of Bojnord,  Bojnord 94531, Iran}
\email{kian@ub.ac.ir}

\renewcommand{\subjclassname}{\textup{2020} Mathematics Subject Classification}
\subjclass[2020]{Primary 46L07, 46L05 Secondary 46L10, 47A64}

\keywords{Hermitian completely bounded maps; completely positive majorants; Wittstock envelopes; Radon--Nikodym theorem;     von Neumann algebras; spectral projections; operator perspectives.}

\begin{abstract}
A Hermitian completely bounded map  $\Phi:\mathscr{A}\to B(\mathcal{H})$
  need not have a canonical completely positive analogue of an absolute value. We study this problem through the ordered set  $\mathsf{E}(\Phi)$
  of completely positive symmetric majorants of  $\Phi$,  which we call Wittstock envelopes.

 We show that every Wittstock envelope canonically determines a minimal
Wittstock envelope below it. However, minimal envelopes  need not be unique. We therefore study the stronger question of when a least completely positive majorant exists.  We obtain a criterion, in terms of central
supports of spectral projections, for the existence of least majorants below a
fixed envelope.  In finite dimensions, the Choi representation gives a simple characterization:
a least completely positive majorant exists if and only if \(\Phi\) or
\(-\Phi\) is completely positive.

Finally, we use the envelope-relative modulus to define envelope-relative
perspectives for Hermitian completely bounded maps.  These constructions extend
Okayasu's completely positive map perspectives, reduce to the usual
perspective in the completely positive case, and remain dependent in general on
the chosen Wittstock envelopes.
\end{abstract}
\maketitle

\section{Introduction}

Completely positive maps form one of the basic ordered structures in operator
algebras, matrix analysis, and quantum information theory. Their order
properties are closely related to Stinespring's dilation theorem
\cite{St}, Arveson's Radon--Nikodym theorem
\cite{Ar}, and, in finite dimensions, the Choi--Jamio{\l}kowski
correspondence \cite{Jami,Choi}. These results provide a
standard framework for studying completely positive maps through positive
operators associated with suitable dilations.

Hermitian completely bounded maps are naturally related to the completely
positive cone through Wittstock-type decomposition results
\cite{Witt,Haag,Paul}. For a Hermitian completely
bounded map \(\Phi\), we consider the ordered set \(\mathsf E(\Phi)\) of
completely positive maps which dominate \(\Phi\) in both signs. The central
question is whether this set contains canonical elements, in particular a least
element which could serve as a completely positive analogue of an absolute
value.

Radon--Nikodym type methods for completely positive maps have been developed in
several directions. Belavkin and Staszewski \cite{Bel-Sta}
proved a Radon--Nikodym theorem for completely positive maps. Raginsky
\cite{Rag} studied Radon--Nikodym derivatives of quantum operations
and their order intervals. Gheondea and Kavruk studied related questions on
absolute continuity for operator-valued completely positive maps
\cite{Gheo-Kav}, and also the relation between Arveson's
Radon--Nikodym derivative and  similarity  for
completely bounded maps \cite{Gheo-Kav2}. These works provide
part of the background for the present use of Radon--Nikodym derivatives in the
study of completely positive majorants.

The finite-dimensional trace-preserving version of the majorant problem was
studied in \cite{Ki-Del-Kr}, where completely positive
trace-preserving majorants of Hermitian-preserving maps between matrix
algebras were characterized. The present paper removes the trace-preserving
constraint and studies the full order-theoretic set of completely positive
majorants. Thus the focus here is not an affine section of the Choi order
interval, but the order structure of all completely positive majorants.

The decomposition of linear maps into completely positive parts has a longer
history. Hadwin \cite{Hadw} studied dilation and Hahn-decomposition
methods for linear maps. In the finite-dimensional setting, Ando
\cite{An18} studied decompositions of positive maps between matrix
algebras as differences of completely positive, or super-positive, maps.
Related decomposition phenomena also appear in the study of tracial positive
maps and their applications in quantum information
\cite{Dad-Ki-Mos}. Related decomposition results are part of the
classical theory of completely bounded maps
\cite{Witt,Haag,Paul}. These works are part of the broader
background on positive and completely positive decompositions, while the
present paper focuses on the order structure of all symmetric completely
positive majorants of a fixed Hermitian completely bounded map.

This work is also connected with the theory of means and perspectives.
Operator means were introduced by Kubo and Ando \cite{K-A}, and the
geometric mean has a form-theoretic realization due to Pusz and Woronowicz
\cite{Pu-W}. More recently, Frenkel, Mosonyi, Vrana and Weiner
\cite{FMVW} studied weighted Kubo--Ando
geometric means in connection with composite quantum hypothesis testing and
introduced a related superoperator perspective function. Okayasu
\cite{Ok} developed geometric means, parallel sums, and
Lebesgue-type decompositions for completely positive maps on von Neumann
algebras. These works concern completely positive inputs. Here the starting point is a
Hermitian completely bounded map, and the completely positive maps used in the
construction arise from its Wittstock envelopes.

We study the ordered set \(\mathsf E(\Phi)\) of Wittstock envelopes of a
Hermitian completely bounded map \(\Phi\).  We show that every completely positive majorant admits a
canonical reduction to a minimal one, while minimal majorants need not be
unique and a least majorant need not exist. We identify the obstruction to the
existence of a least majorant by means of central supports in the relevant von
Neumann algebra and transfer this criterion to completely positive majorants
using Arveson's Radon--Nikodym theorem. In the Choi setting, this yields a
 sign-definiteness criterion for the existence of a least completely
positive majorant.

Thus a Hermitian completely bounded map need not have an intrinsic completely
positive absolute value. This   motivates the final section, where perspectives
for Hermitian completely bounded maps are formulated relative to a chosen
Wittstock envelope. We define envelope-relative sector perspectives and an
envelope-relative absolute perspective by applying Okayasu's completely
positive map perspective \cite{Ok} to the relative positive parts and relative moduli.

\section{Preliminaries}

Throughout the paper, \(\mathscr A\) denotes a unital \(C^*\)-algebra and
\(\mathcal H\) denotes a Hilbert space. We write \(B(\mathcal H)\) for the
algebra of bounded operators on \(\mathcal H\), and
\(\operatorname{CP}(\mathscr A,B(\mathcal H))\) for the cone of completely
positive maps from \(\mathscr A\) into \(B(\mathcal H)\). Elements of
\(C^*\)-algebras will usually be denoted by lower-case letters, such as
\(a,b,c\), and vectors in Hilbert spaces by \(\xi,\eta,\zeta\).

A map
\(\Phi:\mathscr A\to B(\mathcal H)\) is called Hermitian if
\[
   \Phi(a^*)=\Phi(a)^*,
   \qquad a\in\mathscr A .
\]
If \(\Phi,\Psi:\mathscr A\to B(\mathcal H)\) are Hermitian  maps, we
write
\[
   \Phi\leq_{\mathrm{cp}}\Psi
\]
whenever \(\Psi-\Phi\) is completely positive.

We shall use the following form of Stinespring's theorem. If
\(\Omega\in\operatorname{CP}(\mathscr A,B(\mathcal H))\), then there exist a
Hilbert space \(\mathcal K_\Omega\), a unital representation
\[
   \pi_\Omega:\mathscr A\to B(\mathcal K_\Omega),
\]
and a bounded operator \(V_\Omega:\mathcal H\to\mathcal K_\Omega\) such that
\[
   \Omega(a)=V_\Omega^*\pi_\Omega(a)V_\Omega,
   \qquad a\in\mathscr A .
\]
The triple \((\pi_\Omega,\mathcal K_\Omega,V_\Omega)\) is minimal if
\[
   \mathcal K_\Omega
   =
   \overline{\operatorname{span}}
   \{\pi_\Omega(a)V_\Omega\xi:
   a\in\mathscr A,\ \xi\in\mathcal H\}.
\]
Minimal Stinespring representations are unique up to unitary equivalence.

 Let
\(\Omega,\Theta\in\operatorname{CP}(\mathscr A,B(\mathcal H))\) and suppose
that
\[
   \Theta\leq_{\mathrm{cp}}\Omega .
\]
If
\[
   \Omega(a)=V_\Omega^*\pi_\Omega(a)V_\Omega,
   \qquad a\in\mathscr A,
\]
is a minimal Stinespring representation of \(\Omega\), then    Arveson's Radon--Nikodym theorem  guarantees the  existence of  a
unique positive contraction
\[
   D_\Theta^\Omega\in\pi_\Omega(\mathscr A)'
\]
such that
\[
   \Theta(a)
   =
   V_\Omega^*\pi_\Omega(a)D_\Theta^\Omega V_\Omega,
   \qquad a\in\mathscr A .
\]
The operator \(D_\Theta^\Omega\) is called the Radon--Nikodym derivative of
\(\Theta\) with respect to \(\Omega\).

We shall frequently use the following signed form of the preceding theorem.
Let \(\Phi:\mathscr A\to B(\mathcal H)\) be Hermitian and let $\Omega\in\operatorname{CP}(\mathscr A,B(\mathcal H))$ with $-\Omega\leq \Phi\leq \Omega$.  If
\[
   \Omega(a)=V_\Omega^*\pi_\Omega(a)V_\Omega,
   \qquad a\in\mathscr A,
\]
is a minimal Stinespring representation of \(\Omega\), then there exists a
unique self-adjoint contraction
\[
   D_\Phi^\Omega\in\pi_\Omega(\mathscr A)'
\]
such that
\[
   \Phi(a)
   =
   V_\Omega^*\pi_\Omega(a)D_\Phi^\Omega V_\Omega,
   \qquad a\in\mathscr A .
\]
Indeed, applying Arveson's theorem to
\[
   \frac{\Omega+\Phi}{2}
   \quad\text{and}\quad
   \frac{\Omega-\Phi}{2}
\]
gives positive contractions \(D_+\) and \(D_-\) in
\(\pi_\Omega(\mathscr A)'\) satisfying \(D_++D_-=I\), and then
\[
   D_\Phi^\Omega=D_+-D_- .
\]
We call \(D_\Phi^\Omega\) the signed Radon--Nikodym derivative of \(\Phi\) with
respect to \(\Omega\).

For \(T\in\pi_\Omega(\mathscr A)'\), set
\[
   \Omega_T(a)
   =
   V_\Omega^*\pi_\Omega(a)T V_\Omega,
   \qquad a\in\mathscr A .
\]
If \(T\geq0\), then \(\Omega_T\) is completely positive. Moreover, for
positive operators \(S,T\in\pi_\Omega(\mathscr A)'\),
\[
   S\leq T
   \quad\Longrightarrow\quad
   \Omega_S\leq_{\mathrm{cp}}\Omega_T .
\]
Thus the order interval of completely positive maps dominated by \(\Omega\) is
identified with the positive contractions in the commutant
\(\pi_\Omega(\mathscr A)'\).

We use standard notation for positive and negative parts. If \(x=x^*\) belongs
to a von Neumann algebra \(\mathscr M\), then
\[
   x=x_+-x_-,
   \qquad
   |x|=x_++x_-,
\]
where \(x_+,x_-\in\mathscr M_+\) and \(x_+x_-=0\). We denote the support
projections of \(x_+\) and \(x_-\) by \(s(x_+)\) and \(s(x_-)\), respectively.
If \(p\) is a projection in \(\mathscr M\), its central support is denoted by
\(c(p)\); that is, \(c(p)\) is the smallest central projection in
\(\mathscr M\) which dominates \(p\).

We shall distinguish between minimal and least elements in partially ordered
sets. If \(P\) is a partially ordered set, an element \(m\in P\) is minimal if
there is no \(p\in P\) such that \(p<m\). It is least if
\[
   m\leq p,
   \qquad p\in P .
\]
A least element is necessarily unique, while minimal elements need not be
unique.

Finally, in the finite-dimensional Choi setting, for a linear map
\[
   \Phi:\mathbb M_m\to B(\mathcal H),
\]
we write
\[
   C_\Phi=[\Phi(E_{ij})]_{i,j=1}^m
   \in \mathbb M_m(B(\mathcal H))
   \cong B(\mathcal H^{\oplus m}),
\]
where \(\{E_{ij}\}_{i,j=1}^m\) are the matrix units of \(\mathbb M_m\). The
Choi criterion says that \(\Phi\) is completely positive if and only if $ C_\Phi\geq0$.
Consequently, for Hermitian maps \(\Phi,\Psi:\mathbb M_m\to B(\mathcal H)\),
\[
   \Phi\leq_{\mathrm{cp}}\Psi
   \quad\Longleftrightarrow\quad
   C_\Phi\leq C_\Psi .
\]

\section{Minimal Wittstock envelopes}

Let $\Phi:\mathscr A\to B(\mathcal H)$
 be a Hermitian completely bounded map. We denote by
\[
   \mathsf E(\Phi)
   =
   \left\{
      \Omega\in \operatorname{CP}(\mathscr A,B(\mathcal H)):
      -\Omega\leq_{\mathrm{cp}}\Phi\leq_{\mathrm{cp}}\Omega
   \right\}
\]
the set of completely positive symmetric majorants of \(\Phi\). Elements of
\(\mathsf E(\Phi)\) will be called Wittstock envelopes of \(\Phi\). By
Wittstock's decomposition theorem, \(\mathsf E(\Phi)\) is nonempty.

We begin by showing that every Wittstock envelope canonically determines a
minimal Wittstock envelope below it.

\begin{definition}
Let \(\Omega\in\mathsf E(\Phi)\), and let
\[
   \Omega(a)=V_\Omega^*\pi_\Omega(a)V_\Omega,
   \qquad a\in\mathscr A,
\]
be a minimal Stinespring representation of \(\Omega\). Let
\[
   D_\Phi^\Omega\in\pi_\Omega(\mathscr A)'
\]
be the signed Radon--Nikodym derivative of \(\Phi\) with respect to
\(\Omega\). We define the \(\Omega\)-relative positive and negative parts of
\(\Phi\) by
\[
   \Phi_{\Omega,+}(a)
   :=
   V_\Omega^*\pi_\Omega(a)(D_\Phi^\Omega)_+V_\Omega,
   \qquad
   \Phi_{\Omega,-}(a)
   :=
   V_\Omega^*\pi_\Omega(a)(D_\Phi^\Omega)_-V_\Omega,
\]
and the \(\Omega\)-relative modulus of \(\Phi\) by
\[
   |\Phi|_\Omega(a)
   :=
   V_\Omega^*\pi_\Omega(a)|D_\Phi^\Omega|V_\Omega,
   \qquad a\in\mathscr A .
\]
\end{definition}
Here
\[
   D_\Phi^\Omega
   =
   (D_\Phi^\Omega)_+-(D_\Phi^\Omega)_-,
   \qquad
   |D_\Phi^\Omega|
   =
   (D_\Phi^\Omega)_++(D_\Phi^\Omega)_-
\]
is the Jordan decomposition of \(D_\Phi^\Omega\) in the von Neumann algebra
\(\pi_\Omega(\mathscr A)'\).

The maps \(\Phi_{\Omega,+}\), \(\Phi_{\Omega,-}\), and \(|\Phi|_\Omega\) are
completely positive, and
\[
   \Phi=\Phi_{\Omega,+}-\Phi_{\Omega,-},
   \qquad
   |\Phi|_\Omega=\Phi_{\Omega,+}+\Phi_{\Omega,-}.
\]
These maps are independent of the chosen minimal Stinespring representation of
\(\Omega\), up to the canonical unitary equivalence between minimal
Stinespring representations.

\begin{theorem}
\label{thm-env-red}
Let $\Phi:\mathscr A\to B(\mathcal H)$ be Hermitian and completely bounded, and let
$  \Omega\in\mathsf E(\Phi)$. Then
\[
   |\Phi|_\Omega\in\mathsf E(\Phi)
   \qquad\text{and} \qquad
   |\Phi|_\Omega\leq_{\mathrm{cp}}\Omega.
\]
Moreover, \(|\Phi|_\Omega\) is a minimal element of \(\mathsf E(\Phi)\).
\end{theorem}

\begin{proof}
Let   $ \Omega(\cdot)=V^*\pi(\cdot)V$
be a minimal Stinespring representation of \(\Omega\), and consider    $D:=D_\Phi^\Omega\in\pi(\mathscr A)'$. Then $D$ is Hermitian,  $ -I\leq D\leq I$, and $  \Phi(\cdot)=V^*\pi(\cdot)DV$.
   Since \(|D|\in\pi(\mathscr A)'\) is positive and satisfies \(|D|\leq I\), the
map
\[
   |\Phi|_\Omega(a)=V^*\pi(a)|D|V\qquad a\in\mathscr{A}
\]
is completely positive and satisfies $|\Phi|_\Omega\leq_{\mathrm{cp}}\Omega.$
 Moreover, it follows from $ -|D|\leq D\leq |D|$ that the mappings
\[
   |\Phi|_\Omega+\Phi
   =
   V^*\pi(\,\cdot\,)(|D|+D)V
\]
and
\[
   |\Phi|_\Omega-\Phi
   =
   V^*\pi(\,\cdot\,)(|D|-D)V
\]
are completely positive. Therefore
\[
   -|\Phi|_\Omega\leq_{\mathrm{cp}}\Phi\leq_{\mathrm{cp}}|\Phi|_\Omega,
\]
and consequently
\[
   |\Phi|_\Omega\in\mathsf E(\Phi).
\]
We show the minimality of $|\Phi|_\Omega$.   Let
 $ \Gamma\in\mathsf E(\Phi)$ satisfy $ \Gamma\leq_{\mathrm{cp}}|\Phi|_\Omega.$
Since \(|\Phi|_\Omega\leq_{\mathrm{cp}}\Omega\), we also have $   \Gamma\leq_{\mathrm{cp}}\Omega.$ Let $  C:=D_\Gamma^\Omega\in\pi(\mathscr A)'$ be the Radon--Nikodym derivative of \(\Gamma\) with respect to \(\Omega\). Since $ \Gamma\leq_{\mathrm{cp}}|\Phi|_\Omega$, the Arveson’s Radon--Nikodym theorem gives $   0\leq C\leq |D|$.

Because \(\Gamma\in\mathsf E(\Phi)\), the maps $\Gamma+\Phi$  and $\Gamma-\Phi$ are completely positive and so
\[
   0\leq_{\mathrm{cp}}\frac{\Gamma+\Phi}{2}\leq_{\mathrm{cp}}\Gamma
   \leq_{\mathrm{cp}}\Omega
\]
and
\[
   0\leq_{\mathrm{cp}}\frac{\Gamma-\Phi}{2}\leq_{\mathrm{cp}}\Gamma
   \leq_{\mathrm{cp}}\Omega.
\]
Hence Arveson's Radon--Nikodym theorem applies to both
\((\Gamma+\Phi)/2\) and \((\Gamma-\Phi)/2\) with respect to \(\Omega\). In the
fixed Stinespring representation of \(\Omega\), these maps are represented by
\[
   \frac{C+D}{2}
   \qquad\text{and}\qquad
   \frac{C-D}{2},
\]
respectively. By uniqueness of the Radon--Nikodym derivative, these operators
are positive. Thus
\[
   C+D\geq0,
   \qquad
   C-D\geq0,
\]
or equivalently
\[
   -C\leq D\leq C.
\]
Let \(E_D\) denote the spectral measure of \(D\), and put
\[
   p_+=E_D((0,\infty)),\qquad
   p_-=E_D((-\infty,0)),\qquad
   p_0=E_D(\{0\}).
\] These projections belong to
\(\pi(\mathscr A)'\) and satisfy
\[
   p_++p_-+p_0=I.
\]

On the positive spectral subspace of \(D\), we have
\[
   p_+Dp_+=p_+|D|p_+.
\]
Since \(C-D\geq0\) and \(C\leq |D|\), compression to \(p_+\) gives
\[
   p_+Cp_+\geq p_+Dp_+=p_+|D|p_+\qquad \text{and}\qquad p_+Cp_+\leq p_+|D|p_+.
\]
Hence
\[
   p_+Cp_+=p_+|D|p_+.
\]
Similarly, on the negative spectral subspace of \(D\), we obtain
\[
   p_-Cp_-=p_-|D|p_-.
\]
Moreover, it follows from  $0\leq p_0Cp_0\leq p_0|D|p_0=0$,  that $ p_0Cp_0=0$.
Consequently,
\[
   p_+Tp_+=p_-Tp_-=p_0Tp_0=0
\]
holds for positive operator $T:=|D|-C$.
Since \(T\geq 0\) and \(p_+Tp_+=0\), we have
\[
   (T^{1/2}p_+)^*(T^{1/2}p_+)=p_+Tp_+=0.
\]
Thus \(T^{1/2}p_+=0\), and hence \(Tp_+=0\). Taking adjoints gives
\(p_+T=0\).  Similarly \(Tp_-=Tp_0=0\).  Using  $p_++p_-+p_0=I$, we obtain $T=0$, meaning  $C=|D|$.
By uniqueness in Arveson's Radon--Nikodym theorem,
\[
   \Gamma=|\Phi|_\Omega.
\]
Therefore \(|\Phi|_\Omega\) is minimal in \(\mathsf E(\Phi)\).
\end{proof}


Theorem~\ref{thm-env-red}  gives a minimal envelope below any prescribed Wittstock
envelope. We now identify exactly when the prescribed envelope was already
minimal.

\begin{corollary}
\label{cor-min-char}
Let  $\Phi:\mathscr A\to B(\mathcal H)$ be Hermitian and completely bounded, and let $ \Omega\in\mathsf E(\Phi)$. Let $D_\Phi^\Omega\in\pi_\Omega(\mathscr A)'$ be the signed Radon--Nikodym derivative of \(\Phi\) with respect to $\Omega$, computed in a minimal Stinespring representation of \(\Omega\).
Then \(\Omega\) is minimal in \(\mathsf E(\Phi)\) if and only if  $|D_\Phi^\Omega|=I_{\mathcal K_\Omega}$.
\end{corollary}

\begin{proof}
By Theorem~\ref{thm-env-red}, we have
\[
   |\Phi|_\Omega\in\mathsf E(\Phi),
   \qquad
   |\Phi|_\Omega\leq_{\mathrm{cp}}\Omega,
\]
and \(|\Phi|_\Omega\) is minimal.  If \(\Omega\) is minimal in \(\mathsf E(\Phi)\), then the inequality $   |\Phi|_\Omega\leq_{\mathrm{cp}}\Omega$ implies that  $|\Phi|_\Omega=\Omega$.
By the uniqueness part of Arveson's Radon--Nikodym theorem, this is equivalent
to $|D_\Phi^\Omega|=I_{\mathcal K_\Omega}$.

Conversely, if $|D_\Phi^\Omega|=I_{\mathcal K_\Omega}$, then  $|\Phi|_\Omega=\Omega$. However,  Theorem~\ref{thm-env-red} says that \(|\Phi|_\Omega\) is
minimal.
Finally, since \(D_\Phi^\Omega\) is self-adjoint, the condition
\(|D_\Phi^\Omega|=I_{\mathcal K_\Omega}\) is equivalent, by functional
calculus, to \((D_\Phi^\Omega)^2=I_{\mathcal K_\Omega}\).
\end{proof}

The next example shows that minimal Wittstock envelopes need not be unique.
Thus the reduction above is still relative to the initial choice of envelope.

\begin{example}
\label{ex-min-env}
Let \(\mathcal H=\mathbb C^2\), and consider the Hermitian map
\[
   \Phi:\mathbb C\to B(\mathcal H)=\mathbb M_2,
   \qquad
   \Phi(\lambda)=\lambda X,
\]
where
\[
   X=
   \begin{pmatrix}
      1 & 0\\
      0 & -1
   \end{pmatrix}.
\]
For each \(t\in\mathbb R\), put
\[
   Y_t=
   \begin{pmatrix}
      \cosh t & \sinh t\\
      \sinh t & \cosh t
   \end{pmatrix},
   \qquad
   \Omega_t(\lambda)=\lambda Y_t .
\]
Since \(Y_t\geq0\), the map \(\Omega_t\) is completely positive. Moreover,
\[
   Y_t+X=
   \begin{pmatrix}
      \cosh t+1 & \sinh t\\
      \sinh t & \cosh t-1
   \end{pmatrix}
   \quad\text{and}\quad
   Y_t-X=
   \begin{pmatrix}
      \cosh t-1 & \sinh t\\
      \sinh t & \cosh t+1
   \end{pmatrix}
\]
are positive semidefinite, because,  as $2\times 2$ Hermitian matrices, their diagonal entries are nonnegative and
their determinants are equal to  $\cosh^2 t-1-\sinh^2 t=0$.
Thus
\[
   -\Omega_t\leq_{\mathrm{cp}}\Phi\leq_{\mathrm{cp}}\Omega_t,
\]
and hence \(\Omega_t\in\mathsf E(\Phi)\).

We claim that each \(\Omega_t\) is minimal in \(\mathsf E(\Phi)\).
Since \(Y_t>0\), the following is a minimal Stinespring representation of
\(\Omega_t\):
\[
   \pi(\lambda)=\lambda I_{\mathcal H},
   \qquad
   V_t=Y_t^{1/2},
   \qquad
   \Omega_t(\lambda)=V_t^*\pi(\lambda)V_t .
\]
With respect to this representation, the signed Radon--Nikodym derivative of
\(\Phi\) with respect to \(\Omega_t\) is
\[
   D_\Phi^{\Omega_t}=Y_t^{-1/2}XY_t^{-1/2}.
\]
Since
\[
   Y_t^{-1}=
   \begin{pmatrix}
      \cosh t & -\sinh t\\
      -\sinh t & \cosh t
   \end{pmatrix}
   =
   XY_tX,
\]
we have
\[
   (D_\Phi^{\Omega_t})^2
   =
   Y_t^{-1/2}XY_t^{-1}XY_t^{-1/2}
   =
   Y_t^{-1/2}Y_tY_t^{-1/2}
   =
   I_{\mathcal H}.
\]
Hence $|D_\Phi^{\Omega_t}|=I_{\mathcal H}$.
By Corollary~\ref{cor-min-char}, \(\Omega_t\) is a
minimal element of \(\mathsf E(\Phi)\).

Finally, if \(s\neq t\), then \(Y_s\neq Y_t\), and hence $\Omega_s\neq\Omega_t$. Therefore \(\mathsf E(\Phi)\) contains infinitely many distinct minimal
envelopes.
\end{example}


\section{Least completely positive majorants}

 The previous section shows that every Wittstock envelope can be reduced to a
minimal one. We now ask a stronger local question: given
\(\Omega\in\mathsf E(\Phi)\), when do the Wittstock envelopes below \(\Omega\)
have a least element?

Throughout this section, if \(x=x^*\) is an element of a von Neumann algebra
\(\mathscr M\), we denote by \(E_x\) its spectral measure. Thus
\(E_x(\Delta)\in\mathscr M\) is the spectral projection of \(x\) associated with
a Borel set \(\Delta\subseteq\mathbb R\). For a projection \(p\in\mathscr M\),
we denote by \(c_{\mathscr M}(p)\) its central support in \(\mathscr M\), that is,
the smallest central projection \(z\in Z(\mathscr M)\) such that \(p\leq z\).
When the ambient von Neumann algebra is clear, we simply write \(c(p)\).

We first record the von Neumann algebra order fact needed below.
\begin{proposition}
\label{prop-least-maj}
Let $\mathscr M$ be a von Neumann algebra and let $x=x^*\in\mathscr M$. Define
\[
   \mathcal E(x)
   :=
   \{y\in\mathscr M_+:\ -y\leq x\leq y\}.
\]
Let
\[
   p_+ := E_x((0,\infty)),
   \qquad
   p_- := E_x((-\infty,0))
\]
be the positive and negative spectral projections of
\(x\), respectively.
 Then
$\mathcal E(x)$ has a least element if and only if the central supports of
$p_+$ and $p_-$ are orthogonal.

Equivalently, $\mathcal E(x)$ has a least element if and only if there exists a
central projection $z\in Z(\mathscr M)$ such that
\[
   zx\geq0,
   \qquad
   (1-z)x\leq0.
\]
In this case the least element is
\[
   |x|=zx-(1-z)x .
\]
\end{proposition}

\begin{proof}
We divide the proof into several steps.

\smallskip
\noindent\emph{Step 1. A central sign decomposition gives a least majorant.}
First suppose that there exists a central projection $z\in Z(\mathscr M)$ such
that
\[
   zx\geq0
   \qquad\text{and}\qquad
   (1-z)x\leq0.
\]
Set
\[
   y_0:=zx-(1-z)x .
\]
Then $y_0\geq0$ and  $-y_0\leq x\leq y_0$  and so  $y_0\in\mathcal E(x)$.

Let $y\in\mathcal E(x)$. Since $y\geq x$ and $z$ is a central projection
we have
\begin{align}\label{q-p1}
   zy=zyz\geq zxz=zx.
\end{align}
Similarly,  $y\geq -x$ implies that
\begin{align}\label{q-p2}
   (1-z)y=(1-z)y(1-z)\geq -(1-z)x(1-z)= -(1-z)x.
\end{align}
Adding both sides of  \eqref{q-p1} and \eqref{q-p2} yields
\[
   y\geq zx-(1-z)x=y_0 .
\]
Thus $y_0$ is the least element of $\mathcal E(x)$.  Moreover, since \(z\) is central, \(z\) commutes with \(x\), and
\[
   y_0^2
   =
   \bigl(zx-(1-z)x\bigr)^2
   =
   zx^2+(1-z)x^2
   =
   x^2.
\]
As \(y_0\geq0\), uniqueness of the positive square root gives
\(y_0=(x^2)^{1/2}=|x|\).

\smallskip
\noindent\emph{Step 2. Any least majorant must be \(|x|\).} Conversely, suppose that $\mathcal E(x)$ has a least element $y_0$, so that $ y_0\geq x$ and  $ y_0\geq -x$.
 On the other hand, since   $  |x|\in\mathcal E(x)$  we have   $y_0\leq |x|$.
 Compressing both sides of
\[
   x\leq y_0\leq |x|
\]
to the range of $p_+$ gives
\begin{align}\label{q-pos1}
   p_+y_0p_+=p_+xp_+=p_+|x|p_+.
\end{align}
Similarly, compressing both sides of
\[
   -x\leq y_0\leq |x|
\]
to the range of $p_-$ gives
\begin{align}\label{q-pos2}
   p_-y_0p_-=-p_-xp_-=p_-|x|p_-.
\end{align}
Also, if     $p_0=E_x(\{0\})$, then  $p_0|x|p_0=0$ and so the inequality   $0\leq y_0\leq |x|$  implies that
$   p_0y_0p_0=0$. We put $ T:=|x|-y_0$ so that   $T\geq0$. In addition,  \eqref{q-pos1} and \eqref{q-pos2} imply
\[
   p_+Tp_+=p_-Tp_-=p_0Tp_0=0.
\]
For a positive operator $T$ and a projection $p$, the equality $pTp=0$ implies
$Tp=pT=0$. Applying this to $p_+$, $p_-$, and $p_0$, and using $ p_++p_-+p_0=I$,
we obtain $T=0$. Hence $ y_0=|x|$. Thus, if a least element exists, it must be $|x|$.

\smallskip
\noindent\emph{Step 3. Nonorthogonal central supports prevent \(|x|\) from being least.}  For \(n\geq1\), put
\[
   e_n:=E_x([1/n,\infty))
   \qquad\text{and}\qquad
   f_n:=E_x((-\infty,-1/n]).
\]
Since
\[
   (0,\infty)=\bigcup_{n\geq1}[1/n,\infty)
   \qquad\text{and}\qquad
   (-\infty,0)=\bigcup_{n\geq1}(-\infty,-1/n],
\]
we have
\[
   p_+=\bigvee_{n\geq1}e_n
   \qquad\text{and}\qquad
   p_-=\bigvee_{n\geq1}f_n .
\]
Continuing under the assumption that \(\mathcal E(x)\) has a least element,
suppose  (to get a contradiction) that the central supports of
$p_+$ and $p_-$ are not orthogonal, say \(c(p_+)c(p_-)\neq0\).  Since central support preserves joins, the assumption
\(c(p_+)c(p_-)\neq0\) implies that
\(c(e_n)c(f_m)\neq0\) for some \(n,m\).  Putting
\(\alpha=1/n\) and \(\beta=1/m\), we obtain spectral projections
\[
   p=E_x([\alpha,\infty))
   \qquad\text{and}\qquad
   q=E_x((-\infty,-\beta])
\]
with nonorthogonal central supports. Hence, by the comparison theorem for projections in a von Neumann algebra \cite{KadisonR}, there exist nonzero subprojections  $e\leq p$ and  $f\leq q$ and a partial isometry $v\in\mathscr M$ such that  $ vv^*=e$ and $ v^*v=f$.
Choose a positive real number \(s\). Since \(\alpha,\beta>0\), we have
\[
   s^2<s(2\alpha+s),
   \qquad
   s^2<s(2\beta+s).
\]
Hence we may choose a real number \(t>s\), sufficiently close to \(s\), such
that
\[
   t^2\leq s(2\alpha+s),
   \qquad
   t^2\leq s(2\beta+s).
\]
Define
\begin{align}\label{def-vay}
   y:=|x|+s(e+f)+t(v+v^*) .
\end{align}
We claim that $y\in\mathcal E(x)$.
Since  $e\leq E_x([\alpha,\infty))$,  the spectral theorem gives
\[
   |x|+x=2x_+\geq 2\alpha e .
\]
Therefore
\begin{align}\label{vay}
   y+x
   =
   |x|+x+s(e+f)+t(v+v^*)
   \geq
   2\alpha e+s(e+f)+t(v+v^*).
\end{align}
Since \(e\leq p\), \(f\leq q\), and \(pq=0\), the projections \(e\) and \(f\)
are orthogonal. Together with \(v\), they form a system of \(2\times2\) matrix
units in the corner \((e+f)\mathscr M(e+f)\). Thus the linear span of
\(e,f,v,v^*\) is a \(^*\)-subalgebra with identity \(e+f\), naturally
isomorphic to \(\mathbb M_2\):
\begin{align}\label{mat-ide}
   e\mapsto
   \begin{pmatrix}1&0\\0&0\end{pmatrix},
   \quad
   f\mapsto
   \begin{pmatrix}0&0\\0&1\end{pmatrix},
   \quad
   v\mapsto
   \begin{pmatrix}0&1\\0&0\end{pmatrix}.
\end{align}
Under this identification, the element  $(2\alpha+s)e+sf+t(v+v^*)$ corresponds to the matrix
\[   \begin{pmatrix}
      2\alpha+s & t\\
      t & s
   \end{pmatrix},
\]
which is positive because $t^2\leq s(2\alpha+s)$.  Hence $y+x\geq0$ by \eqref{vay}.

Similarly, consider \(y-x\). Since $f\leq E_x((-\infty,-\beta])$, the spectral theorem gives
\begin{align}\label{abx-x}
   |x|-x=2x_-\geq 2\beta f .
\end{align}
It follows from \eqref{def-vay} and \eqref{abx-x} that
\begin{align}\label{vay-x}
   y-x
   =
   |x|-x+s(e+f)+t(v+v^*)
   \geq
   2\beta f+s(e+f)+t(v+v^*).
\end{align}
Under the  matrix-unit identification \eqref{mat-ide}, the element
\[
   2\beta f+s(e+f)+t(v+v^*)
   =
   se+(2\beta+s)f+t(v+v^*)
\]
corresponds to the scalar matrix
\[
   \begin{pmatrix}
      s & t\\
      t & 2\beta+s
   \end{pmatrix}.
\]
This matrix is positive because $t^2\leq s(2\beta+s)$. Hence
\[
   2\beta f+s(e+f)+t(v+v^*)\geq0,
\]
and consequently \(y-x\geq0\).
Thus
\[
   -y\leq x\leq y.
\]
Moreover, it follows from \eqref{vay} and \eqref{vay-x} that
\[
   y=\frac12\{(y+x)+(y-x)\}\geq0.
\]
Therefore $y\in\mathcal E(x)$.

However,
\[
   y-|x|=s(e+f)+t(v+v^*)
\]
corresponds, under the same matrix-unit identification, to
\[
   \begin{pmatrix}
      s & t\\
      t & s
   \end{pmatrix}.
\]
This matrix has eigenvalue \(s-t<0\). Hence \(y-|x|\) is not positive, and so
 $|x|\nleq y$.  This contradicts the fact that $|x|$ must be the least element of
$\mathcal E(x)$.  Therefore the central supports of $p_+$ and $p_-$ must be orthogonal.

\smallskip
\noindent\emph{Step 4. Orthogonal central supports give the required central projection.} If the central supports of $p_+$ and $p_-$ are orthogonal, then we put $z:=c(p_+)$, whence
\[
   zx\geq0,
   \qquad
   (1-z)x\leq0.
\]
Indeed, \(zp_-=0\) and \((1-z)p_+=0\). Hence \(zx\) is supported only on the
positive spectral part of \(x\), while \((1-z)x\) is supported only on the
negative spectral part and the zero spectral part.
This proves the equivalence with the central-projection formulation and
completes the proof.
\end{proof}

Proposition~\ref{prop-least-maj} concerns all positive majorants of a self-adjoint element. For the application to envelopes below a fixed Wittstock envelope, we also need an upper bound. We now present  the corresponding local version.

\begin{theorem}[Local least majorants in a von Neumann algebra]
\label{thm-loc-op}
Let \(\mathscr M\) be a von Neumann algebra, and let
\(d=d^*\in\mathscr M\) be a self-adjoint contraction.
Define
\[
   \mathcal C(d)
   :=
   \{c\in\mathscr M_+:\ 0\leq c\leq I,\ -c\leq d\leq c\}.
\]
Let \(E_d\) denote the spectral measure of \(d\), and put
\[
   p_+ := E_d((0,1)),
   \qquad
   p_- := E_d((-1,0)).
\]
Then \(\mathcal C(d)\) has a least element if and only if the central supports
of \(p_+\) and \(p_-\) in \(\mathscr M\) are orthogonal.
In this case the least element of \(\mathcal C(d)\) is \(|d|\).
\end{theorem}

\begin{proof}
We divide the proof into three steps.

\smallskip
\noindent\emph{Step 1. The possible least element is forced to be \(|d|\).}
First note that \(|d|\in\mathcal C(d)\). We claim that \(|d|\) is minimal in
\(\mathcal C(d)\). To see this, let $c\in\mathcal C(d)$ and $ c\leq |d|$ and  put
\[
   r_+ := E_d((0,\infty)),
   \qquad
   r_- := E_d((-\infty,0)),
   \qquad
   r_0 := E_d(\{0\}).
\]
Since \(c-d\geq0\) and \(c\leq |d|\), compression to \(r_+\) gives
\[
   r_+cr_+\geq r_+dr_+=r_+|d|r_+
   \qquad\text{and}\qquad
   r_+cr_+\leq r_+|d|r_+.
\]
Thus
\begin{align}\label{pos}
   r_+cr_+=r_+|d|r_+.
\end{align}
Similarly, from \(c+d\geq0\) and \(c\leq |d|\), compression to \(r_-\) gives
\begin{align}\label{neg}
   r_-cr_-=r_-|d|r_-.
\end{align}
Finally,
\[
   0\leq r_0cr_0\leq r_0|d|r_0=0,
\]
so  $r_0cr_0=0$.
It follows from \eqref{pos} and \eqref{neg} that
\[
   r_+Tr_+=r_-Tr_-=r_0Tr_0=0
\]
for positive  operator $ T:=|d|-c$.

For a positive operator \(T\) and a projection \(r\), the equality
\(rTr=0\) implies \(Tr=rT=0\). Applying this to
\(r_+,r_-,r_0\), and using
\[
   r_++r_-+r_0=I,
\]
we obtain \(T=0\). Hence \(c=|d|\) and this implies that  \(|d|\) is minimal in
\(\mathcal C(d)\). In particular, if \(\mathcal C(d)\) has a least element,
then that least element must be \(|d|\).

\smallskip
\noindent\emph{Step 2. Orthogonal central supports imply that \(|d|\) is least.}
Assume that the central supports of
\[
   p_+ = E_d((0,1)),
   \qquad
   p_- = E_d((-1,0))
\]
are orthogonal. Let \(c\in\mathcal C(d)\). We prove that  $ c\geq |d|$.  Assume that
\[
   q_+ := E_d(\{1\}),
   \qquad
   q_- := E_d(\{-1\})
\]
are   the endpoint spectral projections.  Since \(c\leq I\) and \(c-d\geq0\),
we have
\[
   q_+cq_+\geq q_+dq_+=q_+,
   \qquad\text{and}\qquad
   q_+cq_+\leq q_+.
\]
Hence  $q_+cq_+=q_+$. Since \(I-c\geq0\),  this implies that
\[
   (I-c)q_+=q_+(I-c)=0.
\]
Thus
\[
   cq_+=q_+c=q_+.
\]
Similarly, using \(c+d\geq0\) on \(q_-\), we obtain
\[
   cq_-=q_-c=q_-.
\]
Therefore \(q_+\) and \(q_-\) reduce \(c\), and on these endpoint spectral
subspaces \(c\) coincides with \(|d|\). We consider  $  h:=I-q_+-q_-$ so that \(h\) reduces both \(c\) and \(|d|\). We take $z:=c_{\mathscr M}(p_+)$.
By definition of central support, \(z\) is a central projection and
\(p_+\leq z\). Since \(c_{\mathscr M}(p_+)c_{\mathscr M}(p_-)=0\), we also have
\[
   zp_-=0,
   \qquad
   (I-z)p_+=0.
\]
 Since
\(hq_-=0\) and \(zp_-=0\), we have
\begin{align}\label{zhp}
   zh(q_-+p_-)=0.
\end{align}
But $q_-+p_-=E_d([-1,0))$ is The negative spectral projection of \(d\). Hence \eqref{zhp} implies that the compression of \(d\) to the projection \(zh\) has no negative
spectral part. Hence
\[
   zhdh\geq0
   \qquad\text{and}\qquad
   zh|d|h=zhdh .
\]
Since \(c-d\geq0\), compression by the projection \(zh\) gives
\begin{align}\label{h1}
   zhc h z\geq zhdhz=zh|d|hz.
\end{align}
Similarly, since
\(hq_+=0\) and \((I-z)p_+=0\), we have
\begin{align}\label{Izhq}
  (I-z)h(q_++p_+)=0.
\end{align}
But, \((q_++p_+)\) is the positive spectral projection of \(d\). Hence, \eqref{Izhq} implies that  \(r:=(I-z)h\) is supported on the nonpositive spectral part of \(d\). Since \(r\)
commutes with \(d\), it follows that
\[
   rdr\leq0
   \qquad\text{and}\qquad
   r|d|r=-rdr .
\]
As \(c+d\geq0\), compression by \(r\) gives
\[
   rcr\geq -rdr=r|d|r .
\]
That is,
\begin{align}\label{h2}
   (I-z)hch(I-z)\geq (I-z)h|d|h(I-z).
\end{align}
Summing both sides of \eqref{h1} and \eqref{h2} gives
\begin{align}\label{h3}
   hch\geq h|d|h .
\end{align}
Since \(q_+\) and \(q_-\) reduce \(c\), the projection
\(h=I-q_+-q_-\) also reduces \(c\). These projections also reduce \(|d|\),
because they are spectral projections of \(d\). Hence
\[
   c-|d|
   =
   q_+(c-|d|)q_+
   +
   q_-(c-|d|)q_-
   +
   h(c-|d|)h .
\]
The first two terms are zero, and the last one is positive by \eqref{h3}. Therefore \(c-|d|\geq0\), that is,
\(c\geq |d|\).
Since \(c\in\mathcal C(d)\) was arbitrary, \(|d|\) is the least element of
\(\mathcal C(d)\).

\smallskip
\noindent\emph{Step 3. Nonorthogonal central supports prevent \(|d|\) from being least.}
Conversely, suppose that the central supports of \(p_+\) and \(p_-\) are not
orthogonal. We show that \(|d|\) is not least in \(\mathcal C(d)\).

For \(n\geq2\), put
\[
   e_n:=E_d([1/n,\,1-1/n])
   \qquad\text{and}\qquad
   f_n:=E_d([-1+1/n,\,-1/n]).
\]
Since
\[
   (0,1)=\bigcup_{n\geq2}[1/n,\,1-1/n]
   \qquad\text{and}\qquad
   (-1,0)=\bigcup_{n\geq2}[-1+1/n,\,-1/n],
\]
we have
\[
   p_+=\bigvee_{n\geq2}e_n
   \qquad\text{and}\qquad
   p_-=\bigvee_{n\geq2}f_n.
\]
Since central support preserves joins, the assumption
\[
   c_{\mathscr M}(p_+)c_{\mathscr M}(p_-)\neq0
\]
implies that
\[
   c_{\mathscr M}(e_n)c_{\mathscr M}(f_m)\neq0
\]
for some \(n,m\geq2\). Put
\[
   \alpha:=1/n,
   \qquad
   \beta:=1/m,
   \qquad
   \delta:=\min\{1/n,1/m\}.
\]
Then the spectral projections  $p:=e_n$  and  $q:=f_m$ satisfy
\[
   p\leq E_d([\alpha,\,1-\delta])
   \qquad\text{and}\qquad
   q\leq E_d([-1+\delta,\,-\beta]),
\]
and have nonorthogonal central supports.

By the comparison theorem for projections in a von Neumann algebra
\cite{KadisonR}, there exist nonzero subprojections $e\leq p$ and $f\leq q$ and a partial isometry \(v\in\mathscr M\) such that
\[
   vv^*=e,
   \qquad
   v^*v=f.
\]
Choose a positive real number \(s\) such that \(0<s<\delta/2\). Since
\(\alpha,\beta>0\), we have
\[
s<\delta-s,\qquad
   s<\sqrt{s(2\alpha+s)},
   \qquad
   s<\sqrt{s(2\beta+s)}.
\]
Thus we may choose a real number \(t\) satisfying
\[
   s<t<
   \min\left\{
      \delta-s,\sqrt{s(2\alpha+s)},\sqrt{s(2\beta+s)}
   \right\}.
\]
Then
\[
   t<\delta-s,
   \qquad
   t^2\leq s(2\alpha+s),
   \qquad
   t^2\leq s(2\beta+s).
\]
Now consider the element
\[
   c:=|d|+s(e+f)+t(v+v^*).
\]
We claim that \(c\in\mathcal C(d)\).

First we prove \(c\leq I\). Since
\[
   e\leq E_d([\alpha,1-\delta])
   \qquad\text{and}\qquad
   f\leq E_d([-1+\delta,-\beta]),
\]
both \(e\) and \(f\) are supported where \(|d|\leq 1-\delta\). Hence
\begin{align}\label{pl1}
   e+f\leq E_{|d|}([0,1-\delta]).
\end{align}
Moreover, by functional calculus we have
\begin{align}\label{pl2}
I-|d|\geq \delta\,E_{|d|}([0,1-\delta]).
\end{align}
It follows from \eqref{pl1} and \eqref{pl2} that
\[
   I-|d|\geq \delta(e+f).
 \]
Together with definition of $c$, this concludes that
\[
   I-c
   =
   I-|d|-s(e+f)-t(v+v^*)
   \geq
   (\delta-s)(e+f)-t(v+v^*).
\]
Since \(e\leq p\), \(f\leq q\), and \(pq=0\), the projections \(e\) and \(f\)
are orthogonal. Together with \(v\), they form a system of \(2\times2\) matrix
units in the corner \((e+f)\mathscr M(e+f)\). Considering the corresponding
identification   with \(\mathbb M_2\), the element
\[
   (\delta-s)(e+f)-t(v+v^*)
\]
is equivalent to
\[
   \begin{pmatrix}
      \delta-s & -t\\
      -t & \delta-s
   \end{pmatrix}.
\]
This matrix is positive because \(t<\delta-s\). Hence $I-c\geq0$  or equivalently, $c\leq I$.

Next, since \(e\leq E_d([\alpha,1-\delta])\), the spectral theorem gives
\[
   |d|+d=2d_+\geq 2\alpha e.
\]
Together with definition of $c$, this implies that
\[
   c+d
   =
   |d|+d+s(e+f)+t(v+v^*)
   \geq
   2\alpha e+s(e+f)+t(v+v^*).
\]
Under the same matrix-unit identification, the element
\[
   2\alpha e+s(e+f)+t(v+v^*)
   =
   (2\alpha+s)e+sf+t(v+v^*)
\]
corresponds to
\[
   \begin{pmatrix}
      2\alpha+s & t\\
      t & s
   \end{pmatrix}.
\]
This matrix is positive because $ t^2\leq s(2\alpha+s)$. Hence $ c+d\geq0$.

Similarly, since \(f\leq E_d([-1+\delta,-\beta])\), the spectral theorem gives
\[
   |d|-d=2d_-\geq 2\beta f.
\]
Together with definition of $c$, this implies that
\[
   c-d
   =
   |d|-d+s(e+f)+t(v+v^*)
   \geq
   2\beta f+s(e+f)+t(v+v^*).
\]
Again, using  the same matrix-unit identification, the element
\[
   2\beta f+s(e+f)+t(v+v^*)
   =
   se+(2\beta+s)f+t(v+v^*)
\]
corresponds to the matrix
\[
   \begin{pmatrix}
      s & t\\
      t & 2\beta+s
   \end{pmatrix},
\]
which is  positive because $t^2\leq s(2\beta+s)$.  Hence $c-d\geq0$.
Consequently,
\[
   c\leq I,
   \qquad
   c+d\geq0,
   \qquad
   c-d\geq0.
\]
Furthermore,  we have
\[
   c=\frac12\{(c+d)+(c-d)\}\geq0.
\]
Hence  $c\in\mathcal C(d)$. However,
\[
   c-|d|=s(e+f)+t(v+v^*)
\]
corresponds, under the same matrix-unit identification, to
\[
   \begin{pmatrix}
      s & t\\
      t & s
   \end{pmatrix}.
\]
This matrix has eigenvalue \(s-t<0\). Hence \(c-|d|\) is not positive, and so
\[
   |d|\nleq c.
\]
Thus \(|d|\) is not a least element of \(\mathcal C(d)\). By Step 1, any least
element would have to be \(|d|\). Therefore \(\mathcal C(d)\) has no least
element.

This completes the proof.
\end{proof}


We now return from the operator-level order problem to completely positive
majorants. Theorem~\ref{thm-loc-op} applies to the Radon--Nikodym derivative of
\(\Phi\) inside a fixed Wittstock envelope, and gives a criterion for the
existence of a least envelope below that fixed envelope.

 For a fixed Wittstock envelope \(\Omega\in\mathsf E(\Phi)\), we write
\[
   \mathsf E_\Omega(\Phi)
   :=
   \{\Gamma\in\mathsf E(\Phi):\ \Gamma\leq_{\mathrm{cp}}\Omega\}.
\]
Thus \(\mathsf E_\Omega(\Phi)\) is the order interval of Wittstock envelopes of
\(\Phi\) lying below the prescribed envelope \(\Omega\).

\begin{theorem}[Least envelopes below a fixed Wittstock envelope]
\label{thm-loc-env}
Let  $\Phi:\mathscr A\to B(\mathcal H)$ be Hermitian and completely bounded, and let
$ \Omega\in\mathsf E(\Phi)$.  Let $\Omega(\cdot)=V_\Omega^*\pi_\Omega(\cdot)V_\Omega$
be a minimal Stinespring representation of \(\Omega\), and let
\[
   D_\Phi^\Omega\in\pi_\Omega(\mathscr A)'
\]
be the signed Radon--Nikodym derivative of \(\Phi\) with respect to
\(\Omega\). Let \(E_{D_\Phi^\Omega}\) denote the spectral measure of
\(D_\Phi^\Omega\) in the von Neumann algebra \(\pi_\Omega(\mathscr A)'\), and put
\[
   p_+ := E_{D_\Phi^\Omega}((0,1)),
   \qquad
   p_- := E_{D_\Phi^\Omega}((-1,0)).
\]
Then \(\mathsf E_\Omega(\Phi)\) has a least element if and only if
the central supports of \(p_+\) and \(p_-\) in
\(\pi_\Omega(\mathscr A)'\) are orthogonal. In this case the least element of
\(\mathsf E_\Omega(\Phi)\) is the \(\Omega\)-relative modulus $|\Phi|_\Omega$.
\end{theorem}

\begin{proof}
We use  $D:=D_\Phi^\Omega$ and $\mathscr M_\Omega:=\pi_\Omega(\mathscr A)'$ for simplicity of notations.  Since \(\Omega\in\mathsf E(\Phi)\), the maps \(\Omega+\Phi\) and \(\Omega-\Phi\) are completely positive. Hence, in terms of the Radon--Nikodym derivative \(D\), we have \[ I+D\geq0, \qquad I-D\geq0. \] Thus \(D=D^*\) and \(-I\leq D\leq I\).
By Arveson's Radon--Nikodym theorem, the order interval $[0,\Omega]
   :=
   \{\Gamma:\ 0\leq_{\mathrm{cp}}\Gamma\leq_{\mathrm{cp}}\Omega\}$
is order isomorphic to the set of positive contractions in
\(\mathscr M_\Omega\). More precisely, a completely positive map
\(\Gamma\leq_{\mathrm{cp}}\Omega\) corresponds to a unique operator
\[
   C:=D_\Gamma^\Omega\in\mathscr M_\Omega,
   \qquad
   0\leq C\leq I,
\]
such that
\[
   \Gamma(a)=V_\Omega^*\pi_\Omega(a)CV_\Omega,
   \qquad a\in\mathscr A .
\]
By definition, a completely positive map \(\Gamma\leq_{\mathrm{cp}}\Omega\)
belongs to \(\mathsf E_\Omega(\Phi)\) precisely when
\[
   \Gamma+\Phi\geq_{\mathrm{cp}}0
   \qquad\text{and}\qquad
   \Gamma-\Phi\geq_{\mathrm{cp}}0.
\]
Since
\[
   (\Gamma+\Phi)(a)
   =
   V_\Omega^*\pi_\Omega(a)(C+D)V_\Omega
   \qquad\text{and}\qquad
   (\Gamma-\Phi)(a)
   =
   V_\Omega^*\pi_\Omega(a)(C-D)V_\Omega,
\]
the Radon--Nikodym theorem identifies complete positivity of these maps
with positivity of \(C+D\) and \(C-D\), respectively.
Equivalently,
\[
   -C\leq D\leq C.
\]

Therefore \(\mathsf E_\Omega(\Phi)\) is order isomorphic to
\[
   \mathcal C(D)
   :=
   \{C\in\mathscr M_{\Omega,+}:\ 0\leq C\leq I,\ -C\leq D\leq C\}.
\]
By Theorem~\ref{thm-loc-op}, this ordered set has a
least element if and only if the spectral projections \[ E_D((0,1)) \qquad\text{and}\qquad E_D((-1,0)) \] have orthogonal central supports in \(\mathscr M_\Omega\). Since \(D=D_\Phi^\Omega\), this is precisely the condition
stated in the theorem.

Moreover, the least element of \(\mathcal C(D)\) is \(|D|\). The corresponding  completely positive map
via  Radon--Nikodym order isomorphism is
\[
   a\longmapsto V_\Omega^*\pi_\Omega(a)|D|V_\Omega,
   \qquad a\in\mathscr A.
\]
By definition, this map is the \(\Omega\)-relative modulus $ |\Phi|_\Omega$.
\end{proof}


In finite dimensions, the criterion in Theorem~\ref{thm-loc-env} gives a
particularly concrete form through the Choi representation.


\begin{corollary}
\label{cor-choi-maj}
Let $\Phi:\mathbb M_m\to B(\mathcal H)$ be Hermitian, and let \(C_\Phi\) be its Choi matrix. Then
\(\mathsf E(\Phi)\) has a least element if and only if
\[
   C_\Phi\geq0
   \qquad\text{or}\qquad
   C_\Phi\leq0.
\]
Equivalently, \(\mathsf E(\Phi)\) has a least element if and only if
\(\Phi\) is completely positive or \(-\Phi\) is completely positive.

In this case the least element is \(\Phi\) when \(\Phi\) is completely
positive, and it is \(-\Phi\) when \(-\Phi\) is completely positive.
Consequently, if \(\Phi\) is neither completely positive nor the negative of a
completely positive map, then \(\mathsf E(\Phi)\) has no least element. In particular, this shows that a general Hermitian completely bounded map admits
no intrinsic completely positive modulus.
\end{corollary}

\begin{proof}
Under the Choi--Jami{\l}kowski isomorphism, the completely positive order on
maps \(\mathbb M_m\to B(\mathcal H)\) is identified with the L\"{o}wner order
on \(\mathbb M_m(B(\mathcal H))\). Hence \(\mathsf E(\Phi)\) is order
isomorphic to
\[
   \{Y\in\mathbb M_m(B(\mathcal H))_+:\ -Y\leq C_\Phi\leq Y\}.
\]
By Proposition~\ref{prop-least-maj}, this ordered set has a least element if
and only if the central supports of the spectral projections
\begin{align}\label{sp-pro}
   E_{C_\Phi}((0,\infty))
   \qquad\text{and}\qquad
   E_{C_\Phi}((-\infty,0))
\end{align}
are orthogonal. However, the von-Neumann algebra
\[
   \mathbb M_m(B(\mathcal H))\cong B(\mathcal H^{\oplus m})
\]
is a factor, so its only central projections are \(0\) and \(I\). Hence the
central support of a projection is \(0\) if the projection is zero, and is
\(I\) otherwise. Therefore the two central supports of \eqref{sp-pro} are orthogonal if
and only if at least one of these two spectral projections is zero. This is
equivalent to
\[
   C_\Phi\geq0
   \qquad\text{or}\qquad
   C_\Phi\leq0.
\]
By Choi's theorem, these   are equivalent to \(\Phi\) being
completely positive or \(-\Phi\) being completely positive, respectively.

If \(C_\Phi\geq0\), the least majorant is \(C_\Phi\), which corresponds to
\(\Phi\). If \(C_\Phi\leq0\), the least majorant is \(-C_\Phi\), which
corresponds to \(-\Phi\). This proves the assertion.
\end{proof}

\begin{remark}
Corollary~\ref{cor-choi-maj} should be read together with
Theorem~\ref{thm-env-red}. Minimal Wittstock envelopes always exist:
 for every \(\Omega\in\mathsf E(\Phi)\), the relative modulus $|\Phi|_\Omega$
is a minimal element of \(\mathsf E(\Phi)\). However, such minimal envelopes
need not be unique and hence need not be comparable. Thus minimality alone does
not produce an intrinsic completely positive modulus depending only on
\(\Phi\).

\end{remark}


\section{Envelope-relative perspectives}

The previous section shows that a Hermitian completely bounded map need not
have an intrinsic completely positive modulus. Thus any perspective construction
which uses such a modulus must be formulated relative to a chosen Wittstock
envelope. In this section we develop that envelope-relative construction.

We use Okayasu's completely positive map perspective. For a scalar function
\(g\) in the admissible class considered in \cite{Ok}, we write
\[
   P_g(\Theta_1,\Theta_2)
\]
for the perspective of completely positive maps
\(\Theta_1,\Theta_2:\mathscr A\to B(\mathcal H)\). Throughout this section,
\(g\) is fixed so that this perspective is well defined.

Let $   \Phi,\Psi:\mathscr A\to B(\mathcal H)$ be Hermitian completely bounded maps, and choose Wittstock envelopes $\Omega\in\mathsf E(\Phi)$ and $ \Lambda\in\mathsf E(\Psi)$.
We write
\[
   \Phi_{\Omega,+},\ \Phi_{\Omega,-}
   \qquad\text{and}\qquad
   \Psi_{\Lambda,+},\ \Psi_{\Lambda,-}
\]
for the corresponding envelope-relative positive and negative parts.

\begin{definition}
Let $\varepsilon,\eta\in\{+,-\}$.
The \emph{\((\Omega,\Lambda)\)-relative sector perspective} of \(\Phi\) and
\(\Psi\) is defined by
\[
    P_g^{\Omega,\Lambda;\varepsilon,\eta}(\Phi,\Psi)
   :=
    P_g(\Phi_{\Omega,\varepsilon},\Psi_{\Lambda,\eta}).
\]
Thus the choice of envelopes \(\Omega\) and \(\Lambda\) determines four
completely positive perspective components, corresponding to the four choices
of signs \((\varepsilon,\eta)\).

If \(\Omega\) and \(\Lambda\) are minimal Wittstock envelopes, we call these
the \emph{minimal-envelope sector perspectives} associated with the chosen
minimal envelopes.
\end{definition}


\begin{definition}
With the same notation, we define the \emph{$(\Omega,\Lambda)$-relative absolute
perspective} of $\Phi$ and $\Psi$ by
\[
      P_g^{\Omega,\Lambda;\mathrm{abs}}(\Phi,\Psi)
   :=
  P_g(|\Phi|_\Omega,|\Psi|_\Lambda).
\]
\end{definition}

\begin{remark}
We do not use an envelope-free notation such as
\[
   P_g(\Phi,\Psi)
\]
for Hermitian completely bounded maps. In general, the maps
\[
   \Phi_{\Omega,+},\quad \Phi_{\Omega,-},\quad |\Phi|_\Omega
\]
depend on the chosen Wittstock envelope \(\Omega\). Consequently, both the
sector perspectives and the absolute perspective depend on the chosen
envelopes. This dependence is unavoidable in general, since a Hermitian completely bounded
map need not have a least completely positive majorant.
\end{remark}


\begin{proposition}
\label{pr-red-invar}
Let $\Phi,\Psi:\mathscr A\to B(\mathcal H)$ be Hermitian completely bounded maps, and let
\[
   \Omega\in\mathsf E(\Phi),
   \qquad
   \Lambda\in\mathsf E(\Psi).
\]
Then, for every \(\varepsilon,\eta\in\{+,-\}\),
\[
   P_g^{\Omega,\Lambda;\varepsilon,\eta}(\Phi,\Psi)
   =
   P_g^{|\Phi|_\Omega,|\Psi|_\Lambda;\varepsilon,\eta}(\Phi,\Psi)
\quad\text{and}\quad
   P_g^{\Omega,\Lambda;\mathrm{abs}}(\Phi,\Psi)
   =
   P_g^{|\Phi|_\Omega,|\Psi|_\Lambda;\mathrm{abs}}(\Phi,\Psi).
\]
\end{proposition}
\begin{proof}
Put $\Omega':=|\Phi|_\Omega$. By Theorem~\ref{thm-env-red}, \(\Omega'\) is a minimal element of
\(\mathsf E(\Phi)\). On the other hand, the relative modulus
\(|\Phi|_{\Omega'}\) is again a Wittstock envelope of \(\Phi\), and it satisfies
\[
   |\Phi|_{\Omega'}\leq_{\mathrm{cp}}\Omega'.
\]
By the minimality of \(\Omega'\), it follows that
\[
   |\Phi|_{\Omega'}=\Omega'=|\Phi|_\Omega .
\]
Since the relative positive and negative parts are defined by
\[
   \Phi_{\Omega,+}:=\frac{|\Phi|_\Omega+\Phi}{2},
   \qquad
   \Phi_{\Omega,-}:=\frac{|\Phi|_\Omega-\Phi}{2},
\]
we get
\[
   \Phi_{\Omega,+}
   =
   \frac{|\Phi|_\Omega+\Phi}{2}
   =
   \frac{|\Phi|_{\Omega'}+\Phi}{2}
   =
   \Phi_{\Omega',+},
\]
and similarly
\[
   \Phi_{\Omega,-}=\Phi_{\Omega',-}.
\]
That is,
\[
   \Phi_{\Omega,+}=\Phi_{|\Phi|_\Omega,+},
   \qquad
   \Phi_{\Omega,-}=\Phi_{|\Phi|_\Omega,-}.
\]

The same argument applied to \(\Psi\) and \(\Lambda\) gives
\[
   \Psi_{\Lambda,+}=\Psi_{|\Psi|_\Lambda,+},
   \qquad
   \Psi_{\Lambda,-}=\Psi_{|\Psi|_\Lambda,-}.
\]

\end{proof}

\begin{example}
\label{ex-pers}
Consider Example~\ref{ex-min-env}:
\[
   \Phi:\mathbb C\to\mathbb M_2,
   \qquad
   \Phi(1)=
   \begin{pmatrix}
      1&0\\
      0&-1
   \end{pmatrix}.
\]
Then for \(r\in\mathbb R\),
\[
   \Omega_r(1)=
   \begin{pmatrix}
      \cosh r&\sinh r\\
      \sinh r&\cosh r
   \end{pmatrix}
\]
is a minimal Wittstock envelope of \(\Phi\). As in the proof of Proposition~\ref{pr-red-invar}, the minimality
of \(\Omega_r\) implies
\[
   |\Phi|_{\Omega_r}=\Omega_r .
\]

Hence the \(\Omega_r\)-relative positive and negative parts are
\[
   \Phi_{\Omega_r,+}
   =
   \frac{\Omega_r+\Phi}{2},
   \qquad
   \Phi_{\Omega_r,-}
   =
   \frac{\Omega_r-\Phi}{2}.
\]
These maps depend on \(r\). Thus the envelope-relative positive and negative
parts are not intrinsic, even when the chosen envelopes are minimal.

The same dependence appears for envelope-relative perspectives. For instance,
for \(g(t)=\sqrt t\), the completely positive map perspective is the geometric
mean, and therefore
\[
   P_{\sqrt{\cdot}}^{\Omega_r,\Omega_r;\mathrm{abs}}(\Phi,\Phi)
   =
   P_{\sqrt{\cdot}}(\Omega_r,\Omega_r)
   =
   \Omega_r .
\]
Consequently, if \(r\neq s\), then
\[
   P_{\sqrt{\cdot}}^{\Omega_r,\Omega_r;\mathrm{abs}}(\Phi,\Phi)
   \neq
   P_{\sqrt{\cdot}}^{\Omega_s,\Omega_s;\mathrm{abs}}(\Phi,\Phi).
\]
Thus envelope-relative perspectives  depend on the chosen
minimal envelope.
\end{example}


We first note that the construction recovers Okayasu's completely positive map
perspective in the positive case. Indeed, when the maps are completely positive
and are used as their own envelopes, the envelope-relative perspective reduces
to Okayasu's perspective.
\begin{proposition}
\label{pr-compat}
Let  $\Phi,\Psi:\mathscr A\to B(\mathcal H)$ be completely positive maps. Consider them as their own Wittstock envelopes, i.e.,
\[
   \Omega=\Phi,
   \qquad
   \Lambda=\Psi .
\]
Then
\[
   P_g^{\Phi,\Psi;+,+}(\Phi,\Psi)
   =
   P_g(\Phi,\Psi)\qquad\text{and}\qquad
   P_g^{\Phi,\Psi;\mathrm{abs}}(\Phi,\Psi)
   =
   P_g(\Phi,\Psi).
\]
\end{proposition}

\begin{proof}
Since \(\Phi\) is completely positive, \(\Phi\in\mathsf E(\Phi)\). With
respect to the envelope \(\Phi\), the Radon--Nikodym derivative of \(\Phi\)
with respect to itself is \(I\). Hence
\[
   |\Phi|_\Phi=\Phi,
   \qquad
   \Phi_{\Phi,+}=\Phi,
   \qquad
   \Phi_{\Phi,-}=0.
\]
Similarly,
\[
   |\Psi|_\Psi=\Psi,
   \qquad
   \Psi_{\Psi,+}=\Psi,
   \qquad
   \Psi_{\Psi,-}=0.
\]
Therefore
\[
   P_g^{\Phi,\Psi;+,+}(\Phi,\Psi)
   =
   P_g(\Phi_{\Phi,+},\Psi_{\Psi,+})
   =
   P_g(\Phi,\Psi).
\]
Also,
\[
   P_g^{\Phi,\Psi;\mathrm{abs}}(\Phi,\Psi)
   =
   P_g(|\Phi|_\Phi,|\Psi|_\Psi)
   =
   P_g(\Phi,\Psi).
\]
\end{proof}


\begin{remark}
When \(g\) is operator monotone, Okayasu's completely positive map perspective
\(P_g\) is monotone in each variable. Hence the envelope-relative sector
perspectives inherit the same monotonicity on each fixed sector. More precisely,
if
\[
   \Phi_{\Omega,\varepsilon}
   \leq_{\mathrm{cp}}
   \Phi'_{\Omega',\varepsilon},
   \qquad
   \Psi_{\Lambda,\eta}
   \leq_{\mathrm{cp}}
   \Psi'_{\Lambda',\eta},
\]
then
\[
   P_g^{\Omega,\Lambda;\varepsilon,\eta}(\Phi,\Psi)
   \leq_{\mathrm{cp}}
   P_g^{\Omega',\Lambda';\varepsilon,\eta}(\Phi',\Psi').
\]
This is simply the monotonicity of \(P_g\) applied to the completely positive
maps
\[
   \Phi_{\Omega,\varepsilon},\quad \Psi_{\Lambda,\eta}.
\]

Moreover, let \(\sigma\) be a Kubo--Ando mean, and
\[
   P_\sigma(\Theta_1,\Theta_2):=\Theta_1\sigma\Theta_2
\]
be the associated completely positive map mean. Since every Kubo--Ando mean is
bounded above by the arithmetic mean, we have
\[
   0\leq \Theta_1\sigma\Theta_2
   \leq \frac12(\Theta_1+\Theta_2)
\]
for completely positive maps \(\Theta_1,\Theta_2\). Applying this to the
relative parts gives, for every \(\varepsilon,\eta\in\{+,-\}\),
\[
   0\leq
   P_\sigma^{\Omega,\Lambda;\varepsilon,\eta}(\Phi,\Psi)
   \leq
   \frac12\bigl(\Phi_{\Omega,\varepsilon}
                +\Psi_{\Lambda,\eta}\bigr).
\]
Similarly, applying the same bound to the relative moduli gives
\[
   0\leq
   P_\sigma^{\Omega,\Lambda;\mathrm{abs}}(\Phi,\Psi)
   \leq
   \frac12\bigl(|\Phi|_\Omega+|\Psi|_\Lambda\bigr).
\]

\end{remark}


\end{document}